\documentclass[11pt]{amsart}
\usepackage{amsfonts, amssymb, amsmath, mathrsfs}
\usepackage{tikz-cd}
\usepackage[hidelinks]{hyperref}
\usepackage{url}
\usepackage[left=25mm,right=25mm,top=38mm,bottom=34mm]{geometry}
\tikzset{
	labl/.style={anchor=south, rotate=90, inner sep=.5mm}
}
\newcommand{\spa}{\medskip}
\usepackage[utf8]{inputenc}

\tikzset{
	labl/.style={anchor=south, rotate=90, inner sep=.5mm}
}
\newcommand{\F}{\mathbb{F}}

\newcommand{\Q}{\mathbb{Q}}
\newcommand{\Z}{\mathbb{Z}}

\newcommand{\Ki}{\Q_p^{\mathrm{ur}}}

\newcommand{\Qpu}{{\mathbb{Q}_p^\mathrm{ur}}}

\newcommand{\Gm}{\mathbb{G}_m}

\newcommand{\calO}{\mathcal{O}}
\newcommand{\calP}{\mathcal{P}}

\newcommand{\calM}{\mathcal{M}}
\newcommand{\calN}{\mathcal{N}}

\newcommand{\calT}{\mathcal{T}}

\newcommand{\mf}{\mathbf}

\newcommand{\Fiisoc}{\mf F^{\infty}\textrm{-} \mf{Isoc}}

\newcommand{\Fnisoc}{\mf F^{n}\textrm{-} \mf{Isoc}}

\newcommand{\Isoc}{\mathbf{Isoc}}
\newcommand{\Isocl}{\mathbf{Isoc}_{\Qpu}}

\newcommand{\LS}{\mathbf{LS}}

\DeclareMathOperator{\Spec}{Spec}

\newcommand{\et}{\textrm{ét}}

\newcommand{\iso}{\xrightarrow{\sim}}

\begin{document}

\newtheorem{theo}[subsection]{Theorem}
\newtheorem*{theo*}{Theorem}
\newtheorem{ques}[subsection]{Question}
\newtheorem*{ques*}{Question}
\newtheorem{conj}[subsection]{Conjecture}
\newtheorem{prop}[subsection]{Proposition}
\newtheorem*{prop*}{Proposition}
\newtheorem{lemm}[subsection]{Lemma}
\newtheorem*{lemm*}{Lemma}
\newtheorem{coro}[subsection]{Corollary}
\newtheorem*{coro*}{Corollary}

\theoremstyle{definition}
\newtheorem{defi}[subsection]{Definition}
\newtheorem*{defi*}{Definition}
\newtheorem{exam}[subsection]{Example}
\newtheorem{rema}[subsection]{Remark}
\newtheorem*{rema*}{Remark}
\newtheorem*{exam*}{Example}
\newtheorem{nota}[subsection]{Notation}
\newtheorem{cons}[subsection]{Construction}

\numberwithin{equation}{subsection}

\title{Slopes of $F$-isocrystals over abelian varieties}
\author{Marco D'Addezio}
\date{\today}

\address{Institut de Mathématiques de Jussieu-Paris Rive Gauche, SU - 4 place Jussieu, Case 247, 75005 Paris}
\email{daddezio@imj-prg.fr}

\begin{abstract}
 We prove that an $F$-isocrystal over an abelian variety defined over a perfect field of positive characteristic has constant slopes. This recovers and extends a theorem of Tsuzuki for abelian varieties over finite fields. Our proof exploits the theory of monodromy groups of convergent isocrystals.
\end{abstract}

\maketitle

\tableofcontents

\section{Introduction}

In this article we chiefly study the behaviour of $F$-isocrystals over abelian varieties. Our main result is the following theorem.

\begin{theo}[Theorem \ref{main:t}]
	\label{i-main:t}
Let $A$ be an abelian variety over a perfect field $k$ of positive characteristic $p$. Every $F$-isocrystal over $A$ has constant slopes\footnote{We say that an $F$-isocrystal $(\calM,\Phi_\calM)$ over a variety $X$ has \textit{constant slopes} if for every closed point $i:x\hookrightarrow X$, the multiset of slopes of $(i^*\calM,i^*\Phi_\calM)$ does not depend on $x$.}.
\end{theo}
Theorem \ref{i-main:t} extends \cite[Thm.~3.7]{TsuNP} and agree with the general expectation that families of smooth projective varieties parametrised by abelian varieties have “small monodromy”, as explained in [\textit{ibid.}]. To prove it we use the theory of \textit{monodromy groups of convergent isocrystals}. This was firstly introduced by Crew in \cite{Cre92} and further studied in \cite{Pal15}, \cite{LP17}, \cite{AD18}, \cite{D'Ad20a}, and \cite{D'Ad20b}. Using this theory, it is possible to prove that the category of convergent isocrystals over $A$, denoted by $\Isoc(A)$, has a rather simple structure. More precisely, we prove the following result.

\begin{prop}[Proposition \ref{comm-fund-group:p}]\label{i-comm-fund-group:p} Let $\Isoc(A)$ be the Tannakian category of convergent isocrystals over $A$. The Tannaka group of $\Isoc(A)$ with respect to any fibre functor is commutative.
\end{prop}

Proposition \ref{i-comm-fund-group:p} is proved using an Eckmann--Hilton argument, exploiting the K\"unneth formula for these Tannaka groups (Proposition \ref{kunneth-formula:p}).  When the ring of Witt vectors of $k$ embeds into the field complex numbers, Proposition \ref{i-comm-fund-group:p} was also obtained independently by Pál in some unpublished notes via a reduction to complex flat connections. 

\spa

Over finite fields, knowing Proposition \ref{i-comm-fund-group:p}, Theorem \ref{i-main:t} follows from a combination of the theory of weights for overconvergent $F$-isocrystals, as developed in \cite{Ked06}, and the $p$-adic global monodromy theorem, proved in \cite[Thm. 4.9]{Cre92} and \cite[Thm. 3.4.4]{D'Ad20a}. In this particular case we can actually prove a stronger result. Write $K$ for the fraction field of the ring of Witt vectors of $k$ and choose an isomorphism $\iota:\bar{K}\iso \mathbb{C}$.

\begin{theo}[Theorem \ref{fini:t}]
	
	\label{i-fini:t}
	If $k$ is a finite field, every $\iota$-pure $F^n$-isocrystal over $A$ becomes constant after passing to a finite étale cover.
	
\end{theo}

  If $k$ is not finite, we cannot rely on the $p$-adic global monodromy theorem, since it is false already for ordinary elliptic curves over $\bar{\F}_p$. To prove Theorem \ref{i-main:t} we reduce instead to the case of generically isoclinic $F$-isocrystals, where the global constancy follows from the semi-continuity of the slope polygon.

\spa
If $X$ is a smooth proper variety over an algebraically closed field $k$, we also deduce from Proposition \ref{i-comm-fund-group:p} a comparison between isocrystals over $X$ and the ones over the Albanese variety $\mathrm{Alb}_X$. Let $\Isoc(X)_F\subseteq \Isoc(X)$ be the subcategory spanned by those convergent isocrystals which can be endowed with a Frobenius structure.

\begin{theo}[Theorem \ref{alb:t}]\label{i-alb:t}
	For every closed point $x\in |X|$, the associated Albanese morphism $f:X \to \mathrm{Alb}_X$ induces a faithfully flat morphism $$\pi_1(\Isoc(X)_F,x)^{\mathrm{ab}}\xrightarrow{f_*} \pi_1(\Isoc(\mathrm{Alb}_X)_F,0_{\mathrm{Alb}_X})$$ of affine group schemes, where $\pi_1(\Isoc(X)_F,x)$ and $\pi_1(\Isoc(\mathrm{Alb}_X)_F,0_{\mathrm{Alb}_X})$ are the Tannaka fundamental groups of $\Isoc(X)_F$ and $\Isoc(\mathrm{Alb}_X)_F$ with respect to $x$ and the identity element $0_{\mathrm{Alb}_X}$. The kernel of $f_*$ is a finite constant group scheme isomorphic to $(\mathrm{Pic}^{\tau}_{X/k}/\mathrm{Pic}^{0,\mathrm{red}}_{X/k})^\vee(k).$
\end{theo}
This theorem is an analogue of
\cite[Thm. 7.1]{Lan12} and \cite[Thm. 4.1]{BdS17}. The main tool we use here, besides Proposition \ref{i-comm-fund-group:p}, is the fact that unit-root $F$-isocrystals correspond to $p$-adic representations of the étale fundamental group of $X$.
\subsection*{Acknowledgements}
I am grateful to Tomoyuki Abe, Gregorio Baldi, Hélène Esnault, Chris Lazda, Ambrus Pál, and Fabio Tonini for many enlightening discussions. I would also like to thank Adrian Langer for his suggestion to apply Proposition \ref{i-comm-fund-group:p} to Albanese varieties, which led to Theorem \ref{i-alb:t}. I thank the organisers and the participants of the workshop “$F$-isocrystals and families of algebraic varieties” at the IMPAN, in Warsaw, for the interest shown in the results of this article. Finally, I thank the anonymous referee for all the comments and corrections which improved the article.

\spa

 The author was funded by the Deutsche Forschungsgemeinschaft (DFG, German Research Foundation) under Germany's Excellence Strategy – The Berlin Mathematics
Research Center MATH+ (EXC-2046/1, project ID: 390685689) and by the Max-Planck Institute for Mathematics.

\section*{Notation}
Let $k$ be a perfect field of positive characteristic $p$ and let $K$ be the fraction field of the ring of Witt vectors of $k$. For a smooth variety $X$ over $k$ we denote by $\Isoc(X)$ the category of $K$-linear convergent isocrystals over $X$, as defined in \cite{Ogu84}. If $X$ is geometrically connected and $\eta$ is a perfect point of $X$, we denote by $\pi_1(\Isoc(X),\eta)$ the Tannaka group of $\Isoc(X)$ with respect to the fibre functor induced by $\eta$ (see \cite[§2.1]{Cre92}). In addition, if $\calM$ is a convergent isocrystal over $X$, we denote by $G(\calM,\eta)$ the Tannaka group of the Tannakian subcategory $\langle \calM \rangle\subseteq \Isoc(X)$, spanned by $\calM$, with respect to the fibre functor induced by $\eta$. We use a similar notation for the other variants of $\Isoc(X)$ that will appear in this article. Also, if $G$ is an affine group scheme, we denote by $G^{\mathrm{ab}}$ the maximal commutative quotient, by $G^{\mathrm{diag}}$ the maximal pro-diagonalisable quotient, and by $G^{\mathrm{uni}}$ the maximal pro-unipotent quotient.

\spa

 Let $F:X\to X$ be the absolute Frobenius of $X$. For a positive integer $n$, we write $\Fnisoc(X)$ for the category of convergent $F^n$-isocrystals\footnote{We recall that by \cite[Thm. 0.7.2]{Ogu90} and \cite[Thm. 2.4.2]{Ber96}, the category $\Fnisoc(X)$ is equivalent to the category of $F^n$-isocrystals over the absolute crystalline site of $X$.} and $\Fiisoc(X)$ for $2\textrm{-}\varinjlim_n \Fnisoc(X)$. If $(\calM,\Phi_\calM)$ is a convergent $F^n$-isocrystal, we write $(\calM,\Phi_\calM^{\infty})$ for its image in $\Fiisoc(X)$. Further, we write $\Isoc(X)_F$ for the smallest strictly full abelian $\otimes$-subcategory of $\Isoc(X)$ closed under subquotients containing all the convergent isocrystals which can be endowed with a Frobenius structure.

 \spa

 	Suppose $k$ algebraically closed. As in \cite[Definition 3.1.2]{D'Ad20b}, we denote by $\Isocl(X,\eta)$ the Tannaka category of convergent isocrystals with \textit{punctual $\Qpu$-structure at $\eta$}. We recall that this category is the category of convergent isocrystal endowed with the choice of a $\Qpu$-linear vector subspace $V_\calM\subseteq \calM_\eta$ such that $V_\calM\otimes_{\Qpu}K(\eta)=\calM_\eta$, where $K(\eta)$ is the fraction field of the ring of Witt vectors of $\Gamma(\eta,\calO_\eta)$. Moreover, we denote by $\Isocl(X,\eta)_F$ the Tannakian subcategory of $\Isocl(X,\eta)$ spanned by the essential image of the functor $\Lambda_\eta:\Fiisoc(X)\to \Isocl(X,\eta)$ constructed in [\textit{ibid.}, Definition 3.1.6].

\section{K\"unneth formula}

In this section we want to prove the K\"unneth formula for the fundamental group of convergent isocrystals. The main ingredient is the following existence theorem.
\begin{theo}[{\cite[§8]{LP17}}\footnote{Note that Theorem \ref{LP:t} can be also obtained as a consequence of \cite{DTZ18} or \cite{Xu19}.}]\label{LP:t}For a smooth morphism $f:Y\to X$ of smooth proper varieties, the functor $f^*:\Isoc(X)\to \Isoc(Y)$ admits a right adjoint $f_*$. The formation of $f_*$ is compatible with base change with respect to morphisms $Z\to X$ where $Z$ is smooth and proper.
\end{theo}
\begin{prop}
	\label{kunneth-formula:p}
	Let $X$ and $Y$ be two smooth proper connected varieties endowed with the choice of rational points $x$ and $y$. The projections of the product $X\times Y$ to the two factors induce an isomorphism $$\pi_1(\Isoc(X\times Y),(x,y))\iso \pi_1(\Isoc(X),x)\times \pi_1(\Isoc(Y),y).$$
\end{prop}

\begin{proof}
	We denote by $q:X\times Y \to X$ the projection to the first factor and by $i$ both $x\hookrightarrow X$ and $x\times Y\hookrightarrow X\times Y$. These morphisms induce the following cartesian diagram
	\begin{equation*}
		\begin{tikzcd}
			 x\times Y \arrow[r,"i",hook] \arrow[rd,phantom] \arrow[d,two heads,"q"]&  X\times Y \arrow[d,two heads, "q"]\\
			x \arrow[r, "i",hook] & X.
		\end{tikzcd}
	\end{equation*}
	Moreover, we get the following sequence of affine group schemes over $K$
	\begin{equation}\label{seq:eq}
		1\to \pi_1(\Isoc(Y),y)\xrightarrow{\alpha} \pi_1(\Isoc(X\times Y),(x,y)) \xrightarrow{\beta} \pi_1(\Isoc(X),x)\to 1,
	\end{equation}
where $\alpha$ is induced by $i^*$ and $\beta$ by $q^*$. We want to use \cite[Thm.~A.13]{DE20} to show that (\ref{seq:eq}) is an exact sequence.

\spa

First, note that the projection $X\times Y\twoheadrightarrow Y$ and the closed immersion $X\times y\hookrightarrow X\times Y$ induce respectively a retraction for $\alpha$ and a section for $\beta$. This shows that $\alpha$ is a closed immersion, $\beta$ is faithfully flat, and $i^*:\Isoc(X\times Y)\to \Isoc(Y)$ is essentially surjective, thus \textit{observable}\footnote{For the definition of an observable functor see \cite[Definition A.2]{DE20}.}. It is also clear by construction that $\beta \circ \alpha$ is trivial.

\spa

 It remains to show that for every convergent isocrystal $\calM$ over $X\times Y$, there exists $\calN\subseteq \calM$, such that $i^*\calN$ is the maximal trivial subobject of $i^*\calM$. We claim that we can take as $\calN$ the convergent isocrystal $q^*q_*\calM$ equipped with the adjunction morphism $q^*q_*\calM\to \calM$. Indeed, by the compatibility of the formation of direct image with base change given by Theorem \ref{LP:t}, we have a natural isomorphism $i^*q^*q_{*}\calM\simeq q^*q_{*}i^*\calM.$ Combining this with the fact that $q_{*}i^*\calM=H^0(Y,i^*\calM)$, we deduce that $i^*q^*q_{*}\calM$ is the maximal trivial subobject of $i^*\calM$. In addition, since $i^*$ is an exact $\otimes$-functor, this also implies that $q^*q_*\calM\to \calM$ is an injective morphism. This concludes the proof of the exactness of (\ref{seq:eq}). For symmetry reasons, we deduce that the analogue sequence where $X$ and $Y$ are exchanged is also exact. Combining these two facts, we get the desired result.
\end{proof}

\begin{rema}
	
	If $X$ and $Y$ are projective one can alternatively recover Proposition \ref{kunneth-formula:p} from \cite[Thm. 7.1]{LP17}. A variant of Proposition \ref{kunneth-formula:p} is also proven in \cite[Thm. III]{DTZ18}.

\end{rema}

\section{Isocrystals with commutative monodromy}
This section is an interlude on convergent isocrystals with commutative monodromy. The main result in this section is that every Frobenius structure on these isocrystals has constant slopes (Proposition \ref{comm-then-cons:p}). As we will see in §\ref{isoc-over-abel:s}, over abelian varieties the monodromy group of a convergent isocrystal is always commutative.

\begin{lemm}\label{irr:l}
	Suppose $k$ algebraically closed and let $(\calM,\Phi_\calM)$ be a convergent $F^n$-isocrystal over $X$. If $(\calM,\Phi_\calM^m)$ is irreducible for every $m>0$, then $\calM$ is irreducible.
\end{lemm}
\begin{proof}
	Let $\calN\subseteq \calM
	$ be an irreducible subobject. By \cite[Corollary 6.2]{Laz17}, the functor $(F^n)^*$ is an autoequivalence of $\Isoc(X)$, thus $(F^n)^*$ permutes the isomorphism classes of the irreducible subobjects of $\calM$. We deduce that after possible replacing $n$ with a multiple, we have that $(F^n)^*\calN\simeq \calN$. In other words, $\calN$ can be endowed with some $F^n$-structure $\Phi_\calN$. Write $(\calP,\Phi_\calP^\infty)$ for $(\calN,\Phi_\calN^\infty)^\vee\otimes (\calM,\Phi_\calM^\infty)$. If $\calT\subseteq \calP$ is the maximal trivial subobject of $\calP$, it defines a subobject $(\calT,\Phi_\calT^\infty)\subseteq (\calP,\Phi_\calP^\infty)$. Up to replacing $\Phi_\calN$ with $p^s\Phi_\calN^r$ for some $(s,r)\in \Z\times\Z_{>0}$, we may assume that one of the slopes of $(\calT,\Phi_\calT^\infty)$ is $0$. Since $(\calT,\Phi_\calT^\infty)$ comes from a convergent $F^\infty$-isocrystal over $\Spec(k)$, we deduce that $(\calT,\Phi_\calT^\infty)$ has a non-trivial global section in $\Fiisoc(X)$. This implies that there exists a non-zero morphism $(\calN,\Phi_\calN^\infty)\to (\calM,\Phi_\calM^\infty)$. Since $(\calM,\Phi_\calM^\infty)$ is irreducible, we deduce that $(\calN,\Phi_\calN^\infty)=(\calM,\Phi_\calM^\infty)$. In turn, this implies that $\calM$ is irreducible, as we wanted.
\end{proof}

\begin{prop}\label{comm-then-cons:p}

	Let $(\calM,\Phi_\calM)$ be a convergent $F^n$-isocrystal over a geometrically connected variety $X$ over a perfect field $k$. If the monodromy group $G(\calM,\eta)$ is commutative for some perfect point $\eta$, then the slopes of $(\calM,\Phi_\calM)$ are constant.
\end{prop}
\begin{proof}
Thanks to \cite[(2.1.10)]{Cre92} we may assume $k$ algebraically closed and we may replace $(\calM,\Phi_\calM)$ with the induced convergent $F^\infty$-isocrystal $(\calM,\Phi_\calM^\infty)$. By looking at the irreducible subquotients, we may further assume that $(\calM,\Phi_\calM^\infty)$ is irreducible. The aim is to show that generically $(\calM,\Phi_\calM^\infty)$ admits a unique slope. Indeed, thanks to \cite[Theorem 3.12]{Ked16}, this would imply that $(\calM,\Phi_\calM^\infty)$ admits a unique slope globally. 

\spa

To prove this we choose a closed point $i:x\hookrightarrow X$ where $(\calM,\Phi_\calM^\infty)$ has the same slopes as the generic ones. Let $i^*:\langle \calM, \Phi^\infty_\calM \rangle\to \Fiisoc(x)$ be the induced restriction functor. By the Dieudonné--Manin classification, the category $\Fiisoc(x)$ is equivalent to the category of $\Q$-graded $\Qpu$-vector spaces, where the $\Q$-graduation is induced by the slopes. Therefore, if $d$ is the lcm of the denominators of the generic slopes of $(\calM,\Phi_\calM)$, the functor $i^*$ induces a morphism $\chi:\Gm^{1/d}\to G(\calM,\Phi^\infty_\calM,x)$, where $\Gm^{1/d}$ is the dimension 1 torus with characters $\tfrac{1}{d}\Z$. Let $V_\calM$ be the Dieudonné--Manin structure of $\calM$ at $x$ associated to $\Phi_\calM$ (cf. \cite[Def. 3.1.6]{D'Ad20b}). By (the proof of) [\textit{ibid.}, Prop. 3.2.8],  the associated monodromy group $G(\calM, V_\calM,x)$ is a normal subgroup of $G(\calM,\Phi^\infty_\calM,x)$. Therefore, the morphism $\chi$ induces an action of $\Gm^{1/d}$ on $G(\calM,V_\calM,x)$ by conjugation. 

\spa

Thanks to [\textit{ibid.}, Prop. 3.3.2], the algebraic group $G(\calM,V_\calM,x)$ is a $\Qpu$-form of $G(\calM,x)$, thus it is commutative by our assumption. In addition, thanks to Lemma \ref{irr:l}, we know that $(\calM,V_\calM)$ is irreducible, which implies that $G(\calM,V_\calM,x)$ is reductive. Since $G(\calM,V_\calM,x)$ is a commutative reductive group, its group scheme of automorphisms is a discrete group. This implies that the action of $\Gm^{1/d}$ on $G(\calM,V_\calM,x)$ must be trivial. If $(\calM,\Phi_\calM)$ had at least two generic slopes, then $V_{\calM}$ would decompose as $V_\calM^{[a]}\oplus W$, where $V_\calM^{[a]}$ is the subspace of $V_\calM$ of slope $a\in \Q$ and $W$ is its Frobenius-stable direct summand. Since $\chi(\Gm^{1/d})$ commutes with $G(\calM,V_\calM,x)$, this decomposition would be stable under the action of $G(\calM,V_\calM,x)$, and thus it would induce a decomposition of $\calM$ in two pieces. This would contradict the fact that $\calM$ is irreducible.
\end{proof}

We end this section with a consequence of Proposition \ref{comm-then-cons:p} that we will need later on.

\begin{coro}\label{diag-isom:c}
	If $k$ is algebraically closed, there is a natural isomorphism
	$$\pi_1(\Isoc_{\Ki}(X,\eta)_F,\eta)^{\mathrm{diag}}\iso\pi_1(\LS(X,\Ki),\eta)^{\mathrm{diag}},$$ where $\LS(X,\Qpu)$ is the category of lisse $\Qpu$-sheaves over $X$.
\end{coro}
\begin{proof}In \cite[Prop. 3.3.4]{D'Ad20b} we constructed a natural fully faithful functor $\Psi:\LS(X,\Ki)\hookrightarrow \Isoc_{\Ki}(X,\eta)_F$. By [\textit{loc. cit.}], the essential image is spanned by those isocrystals with $\Qpu$-structure that can be endowed with an isoclinic Frobenius structure. Therefore, to prove the corollary it is enough to show that every object $(\calM,V_\calM)$ in $\Isoc_{\Ki}(X,\eta)_F$ with diagonalisable monodromy group is in the essential image of $\Psi$.
	
	\spa

	Without loss of generality, we may assume that $(\calM,V_\calM)$ comes from an $F^n$-isocrystal $(\calM,\Phi_\calM)$. In addition, since $(\calM,V_\calM)$ is semi-simple, it is enough to prove that the irreducible subobjects of $(\calM,\Phi_\calM)$ are isoclinic.
	 This simply follows from Proposition \ref{comm-then-cons:p}.\end{proof}
\section{Isocrystals over abelian varieties}
\label{isoc-over-abel:s}
Let $A$ be an abelian variety over $k$ with identity point $0_A$. We want to prove that the $F^n$-isocrystals over $A$ have constant slopes. For this scope, we first prove that the Tannaka group of the category of convergent isocrystals over $A$ is commutative.

\begin{prop}\label{comm-fund-group:p}
	The affine group scheme $\pi_1(\Isoc(A),0_A)$ is commutative.
\end{prop}
\begin{proof}
We want to prove that $\pi_1(\Isoc(A),0_A)$ is commutative using an Eckmann--Hilton argument (see \cite[Thm~5.4.2]{EH}). By Proposition~\ref{kunneth-formula:p}, the two projections of $A\times A$ to its factors induce an isomorphism $$\pi_1(\Isoc(A\times A),0_A\times 0_A)\iso \pi_1(\Isoc(A),0_A)\times \pi_1(\Isoc(A),0_A).$$ If $m:A\times A\to A$ is the multiplication map of $A$, the morphism 
\begin{align*}
	\widetilde{m}_*:\pi_1(\Isoc(A),0_A)\times \pi_1(\Isoc(A),0_A)& \iso \pi_1(\Isoc(A\times A),0_A\times 0_A)\xrightarrow{m_*}\pi_1(\Isoc(A),0_A)
\end{align*}
endows $\pi_1(\Isoc(A),0_A)$ with the structure of a group object in the category of affine group schemes. This implies that $\pi_1(\Isoc(A),0_A)$ is commutative, as we wanted.
\end{proof}

\begin{theo}\label{main:t}
If $A$ is an abelian variety over a perfect field $k$ of positive characteristic, every $F^n$-isocrystal over $A$ has constant slopes.
\end{theo}
\begin{proof}
Let $(\calM,\Phi_\calM)$ be an $F^n$-isocrystal over $A$. By Proposition \ref{comm-fund-group:p}, the monodromy group $G(\calM,0_A)$, being a quotient of $\pi_1(\Isoc(A),0_A)$, is commutative. Thanks to Proposition \ref{comm-then-cons:p}, we deduce that the slopes of $(\calM,\Phi_\calM)$ are constant. This ends the proof.
\end{proof}

\begin{theo}

	\label{fini:t}
	If $k$ is a finite field, every $\iota$-pure $F^n$-isocrystal over $A$ becomes constant after passing to a finite étale cover.

\end{theo}
\begin{proof}Without loss of generality we may assume that $n=[k:\F_p]$.
By \cite[Cor. 3.5.2]{D'Ad20a}, if $(\calM,\Phi_\calM)$ is a $\iota$-pure $F^n$-isocrystal over $A$, then $\calM$ is semi-simple\footnote{In the notation of \cite{D'Ad20a}, the $F^n$-isocrystal $(\calM,\Phi_\calM)$ is a \textit{$K$-coefficient object} and $\calM$ is the \textit{geometric $K$-coefficient object} associated to $\calM$.}. As a consequence, thanks to [\textit{ibid.}, Cor. 3.4.5], the neutral component $G(\calM,\eta)^\circ$ is a semi-simple algebraic group. Combining this with Proposition \ref{comm-fund-group:p}, we deduce that $G(\calM,\eta)^\circ$ is trivial. Therefore, by [\textit{ibid.}, Prop. 3.3.4], after passing to a finite étale cover of $A$, the isocrystal $\calM$ becomes trivial. This yields the desired result.
\end{proof}

To end the article, we want to prove an additional consequence of Proposition \ref{comm-fund-group:p}, which is an analogue of \cite{Lan12} and \cite{BdS17}.

\begin{theo}\label{alb:t}
	Let $X$ be a smooth connected proper variety over an algebraically closed field $k$ of positive characteristic and let $x$ be a $k$-point of $X$. If $f:X \to \mathrm{Alb}_X$ is the Albanese morphism mapping $x$ to $0_{\mathrm{Alb}_X}$, the induced morphism $$\pi_1(\Isoc(X)_F,x)^{\mathrm{ab}}\xrightarrow{f_*} \pi_1(\Isoc(\mathrm{Alb}_X)_F,0_{\mathrm{Alb}_X})$$ is faithfully flat. Moreover, the kernel of $f_*$ is a finite constant group scheme over $K$ isomorphic to $C:=(\mathrm{Pic}^{\tau}_{X/k}/\mathrm{Pic}^{0,\mathrm{red}}_{X/k})^\vee(k).$
\end{theo}

\begin{proof}
	Write $G$ for $\pi_1(\Isoc(X)_F,x)$ and $H$ for $\pi_1(\Isoc(\mathrm{Alb}_X)_F,0_{\mathrm{Alb}_X})$. By Proposition \ref{comm-fund-group:p}, the affine group scheme $H$ is commutative, therefore both $G^{\mathrm{ab}}$ and $H$ decompose as a product of a pro-diagonalisable affine group and a commutative pro-unipotent affine group. By \cite[Lem. 5]{KL81}, the morphism $\pi_1^\et(X,x)^{\mathrm{ab}}\to\pi_1^\et(\mathrm{Alb}_X,{0}_{\mathrm{Alb}_X})$ is surjective and the kernel is isomorphic to $C$. This implies that $\pi_1(\LS(X,\Ki),x)^{\mathrm{ab}}\to\pi_1(\LS(\mathrm{Alb}_X,\Ki),0_{\mathrm{Alb}_X})$
	is faithfully flat with kernel $C$. By Corollary \ref{diag-isom:c} we deduce then that $$\pi_1(\Isoc_{\Ki}(X,x)_F,x)^{\mathrm{diag}}\to\pi_1(\Isoc_{\Ki}(\mathrm{Alb}_X,0_{\mathrm{Alb}_X})_F,0_{\mathrm{Alb}_X})^{\mathrm{diag}}$$
	 is faithfully flat with kernel $C$. Finally, by virtue of \cite[Prop. 3.3.2]{D'Ad20b}, we get that $G^{\mathrm{diag}}\to H^{\mathrm{diag}}$ is faithfully flat with kernel $C$. 
	
	\spa
	
	 It remains to prove that the morphism $G^{\mathrm{ab, uni}}\to H^{\mathrm{uni}}$ is an isomorphism. By \cite[Thm. 5.4]{DE20} and its proof, the category $\Isoc(X)_F$ (resp. $\Isoc(\mathrm{Alb}_X)_F$) is a Serre subcategory of $\Isoc(X)$ (resp. $\Isoc(\mathrm{Alb}_X)$). This implies that every convergent isocrystal over $X$ with unipotent monodromy is contained in $\Isoc(X)_F$. Therefore, the affine group $G^{\mathrm{uni}}$ (resp. $H^{\mathrm{uni}}$) is the Tannaka group of the category of unipotent convergent isocrystals over $X$ (resp. $\mathrm{Alb}_X$). In other words, the affine group scheme $G^{\mathrm{uni}}$ (resp. $H^{\mathrm{uni}}$) coincides with the fundamental group of $X$ (resp. $\mathrm{Alb}_X$) considered in \cite[§2.2.1]{CLS99}. As explained in the proof of [\textit{ibid.}, Prop. 3.2.1], the Lie algebra of $G^\mathrm{ab, uni}$ (resp. $H^{\mathrm{uni}}$) is dual to $H^1_{\mathrm{rig}}(X)$ (resp. $H^1_{\mathrm{rig}}(\mathrm{Alb}_X))$. Thanks to \cite[Rmq. II.3.11.2]{Ill79}, we deduce that $G^{\mathrm{ab, uni}}\to H^{\mathrm{uni}}$ is an isomorphism, as we wanted.
\end{proof}
\bibliographystyle{ams-alpha}

\end{document}